\documentclass[reqno,a4paper]{amsart}

\usepackage[english]{babel}
\usepackage{amsmath,amsthm,amssymb,amsfonts, comment}
\usepackage{mathrsfs}

\usepackage{scrtime}
\usepackage{enumitem}
\setlist[enumerate]{leftmargin=*}

\usepackage{hyperref}

\numberwithin{equation}{section}   %%numera le equazioni sezione per sezione

\newtheorem{theorem}{Theorem}[section]
\newtheorem{lemma}[theorem]{Lemma}
\newtheorem{proposition}[theorem]{Proposition}
\newtheorem{corollary}[theorem]{Corollary}

\theoremstyle{definition}

\newcommand{\spnt}[2]{\left\lfloor #1, #2 \right\rceil}

\newcommand{\tc}{\,:\,}

\newcommand{\RR}{\mathbb{R}}

\newcommand{\NN}{\mathbb{N}}

\newcommand{\DimD}{d}
\newcommand{\DimK}{k}                 % dimension of the equatorial sphere

\newcommand{\defeq}{\mathrel{:=}}

\newcommand{\Harm}{\mathcal{H}}

\newcommand{\sfera}{\mathbb{S}}   % sphere

\newcommand{\group}[1]{\mathrm{#1}}  
\newcommand{\meas}{\sigma}        % measure on sphere
\newcommand{\bigO}{\mathcal{O}}   % Landau notation

\newcommand{\LegF}{\mathrm{P}}     % associated Legendre (Ferrers) function

\newcommand{\tX}{X}      % modified spherical harmonic S^d 

\makeatletter
\def\author@andify{%
  \nxandlist {\unskip ,\penalty-1 \space\ignorespaces}%
    {\unskip {} \@@and~}%
    {\unskip \penalty-2 \space \@@and~}%
}
\makeatother

\begin{document}

\title[Estimates for ultraspherical polynomials]{Uniform pointwise estimates \\ for ultraspherical polynomials}
\author{Valentina Casarino}
\address{Universit\`a degli Studi di Padova\\Stradella san Nicola 3 \\I-36100 Vicenza \\ Italy}
\email{valentina.casarino@unipd.it}
\author{Paolo Ciatti}
\address{Universit\`a degli Studi di Padova\\Via Marzolo 9 \\I-35100 Padova \\ Italy}
\email{paolo.ciatti@unipd.it}
\author{Alessio Martini}
\address{School of Mathematics \\ University of Birmingham \\ Edgbaston \\ Birmingham \\ B15 2TT \\ United Kingdom}
\email{a.martini@bham.ac.uk}

\begin{abstract}
We prove pointwise bounds for two-parameter families of Jacobi polynomials.
Our bounds 
imply estimates 
for a class of functions   
 arising from the spectral analysis of distinguished Laplacians and sub-Laplacians on the unit sphere in arbitrary dimension, and are instrumental in the proof of sharp multiplier theorems for those operators.
\end{abstract}

\thanks{
The first and the second author were partially supported by GNAMPA (Project 2018
``Operatori e disuguaglianze integrali in spazi con simmetrie")
and MIUR (PRIN 2016 ``Real and Complex Manifolds: Geometry, Topology and Harmonic Analysis").
Part of this research was carried out while the third author was visiting the Dicea, Universit\`a di Padova,  Italy,  as a recipient of a ``Visiting Scientist 2019'' grant;
he gratefully thanks the Universit\`a di Padova for the support and hospitality. 
The authors are members of the Gruppo Nazionale per l'Analisi Matematica, la Probabilit\`a e le loro Applicazioni (GNAMPA) of the Istituto Nazionale di Alta Matematica (INdAM)}

\maketitle

\section{Introduction}\label{s:intro}
The primary purpose of this work is to prove  pointwise  estimates for a family of functions that
 are fundamentally related to the spectral analysis  of spherical Laplacians and sub-Laplacians and expressed in terms of ultraspherical polynomials.
More specifically, for a fixed  $\DimD \in \NN$, $\DimD\geq 2$, 
we consider  the functions
\begin{equation}\label{eq:Xtilde}
\tX_{\ell, m}^\DimD (x)
= c_{\ell m} (1-x^2)^{m/2-(\DimD-2)/4} P_{\ell-m-1/2}^{(m,m)}(x).
\end{equation}
Here 
$\ell\in \NN_d \defeq \NN + (d-1)/2$, $m\in  \NN_{d-1}$,
 $m\leq \ell$,   
$x \in [-1,1]$,
 the symbol $P^{(\alpha,\beta)}_j$  denotes the Jacobi polynomial of degree $j \in \NN$ and indices $\alpha, \beta>-1$, and
$c_{\ell m}$ 
is the normalization constant given by
\begin{equation}\label{eq:clm}
c_{\ell m} = \frac{\bigl[\ell \, \Gamma ( \ell-m+1/2)\, \Gamma (\ell +m+1/2)\bigr]^{1/2}}{2^m \,\Gamma( \ell+1/2)}
\end{equation}
and chosen so that 
\begin{equation}\label{eq:Xlm_def_iniziale-norm}
\int_{-1}^1
|\tX_{\ell, m}^\DimD (x)|^2 \,(1-x^2)^{(\DimD-2)/2} \,dx = 1,
\end{equation}
see \cite[(4.3.3)]{Szego}.

The functions $\tX_{\ell, m}^\DimD$ are instrumental in the recursive construction of orthonormal bases of 
$L^2(\sfera^\DimD)$, $\sfera^\DimD$ denoting the unit sphere in $\RR^{1+\DimD}$, made of spherical harmonics. Namely, for all $k \geq 1$ and $m \in \NN_k$,
let $\Harm^m(\sfera^\DimK)$ denote the space of spherical harmonics (that is, restrictions to the spherical surface of harmonic polynomials) of degree $m-(\DimK-1)/2$ on the unit sphere in $\RR^{1+\DimK}$. Moreover, for all functions $f$ on $\sfera^{\DimD-1}$, let us define the function $\tX^{\DimD}_{\ell,m} \otimes f$ on $\sfera^d$ by
\[
(\tX^{\DimD}_{\ell,m} \otimes f) ((\cos \psi) \omega, \sin \psi) = \tX^{\DimD}_{\ell,m}(\sin \psi) f(\omega),
\]
for all $\omega \in \sfera^{\DimD-1}$ and $\psi \in [-\pi/2,\pi/2]$ (this definition makes sense almost everywhere on $\sfera^d$; actually, when $m>(d-2)/2$, it makes sense everywhere, because $\tX^{\DimD}_{\ell,m}(\pm 1) = 0$ in that case). Then, for all 
$\ell \in \NN_d$
and 
$m \in \NN_{d-1}$
such that  $m\leq \ell$, the map $f \mapsto \tX^{\DimD}_{\ell,m} \otimes f$ is an isometric embedding of $\Harm^m(\sfera^{\DimD-1})$ into $\Harm^\ell(\sfera^{\DimD})$ (with respect to the Hilbert space structures induced by $L^2(\sfera^{\DimD-1})$ and $L^2(\sfera^\DimD)$ respectively), and indeed we have the orthogonal direct sum decomposition
\begin{equation}\label{eq:decomp-Hl}
\Harm^\ell(\sfera^{\DimD}) = \bigoplus_{m \leq \ell} \tX^{\DimD}_{\ell,m} \otimes \Harm^m(\sfera^{\DimD-1}).
\end{equation}
This construction is classical and can be found in several places in the literature, modulo some minor notational differences
 (see, e.g., \cite[Ch.\ IX]{Vilenkin} or \cite[Chapter XI]{EMOT}).

In order to obtain pointwise estimates for $\tX_{\ell, m}^\DimD (x)$, it is natural to seek bounds for
 the ($\DimD$-independent) functions 
\[
Y_{\ell,m}(x) = c_{\ell,m} (1-x^2)^{m/2} \, P_{\ell-m-1/2}^{(m,m)}(x),
\]
  with $ (\ell,m) \in (\NN/2)^2$ and $ \ell-m-1/2 \in \NN$.
Upper bounds for Jacobi polynomials $P^{(\alpha,\beta)}_j$, that are uniform with respect to $\alpha$, $\beta$ and $j$ in suitable ranges, have recently attracted a considerable interest. 
For a brief account of these bounds, with particular emphasis on  Bernstein-type inequalities,  we refer to  \cite{EMN};
for some earlier results on ultraspherical polynomials and the strictly related  
associated Legendre functions,
see  \cite{Lo1, Lo2}.
For recent contributions, focusing on  the uniformity with respect to the indices,
 we refer to  works of  Haagerup and Schlichtkrull  \cite{Haagerup}, Koornwinder, Kostenko and Teschl \cite{KKT}, and Krasikov \cite{Kra2}. 
In the particular case $d=2$, some relevant upper bounds for the classical spherical harmonics
 may be found in \cite{Ward, BDWZ, Frank}.

Most of the aforementioned results give uniform weighted estimates for suitably normalised families of Jacobi polynomials $P_{j}^{(\alpha,\beta)}$, where the weight depends on the type $(\alpha,\beta)$ and is independent of the degree $j$. In contrast, the estimates that we obtain here take into consideration, for each individual function $Y_{\ell,m}$, the position of the ``transition points'' $\pm a_{\ell,m}$ (see \eqref{eq:def-aldm} below) that separate the regions of oscillation and decay of $Y_{\ell,m}$ on $[-1,1]$. Estimates of this nature, that describe with a certain precision the behaviour of the function near the transition points, turn out to be essential ingredients in the proof of a sharp spectral multiplier theorem for Grushin operators on the unit sphere $\sfera^\DimD$, whose spectral decomposition can be expressed in terms of spherical harmonics. In the case $\DimD=2$, this problem was studied in \cite{CaCiaMa}, where pointwise estimates of this type were proved for the functions $\tX_{\ell,m}^2$. The present paper confirms the validity of similar estimates for the functions $\tX_{\ell,m}^\DimD$ with arbitrary $\DimD \geq 2$; details on their application to the proof of a multiplier theorem are given in \cite{CaCiaMa2}. We also refer to \cite[Section 8]{HoMa} for the discussion of estimates of this kind for a different family of Jacobi polynomials (namely, $P^{(\alpha,\beta)}_j$, with $\alpha\neq \beta$ and only one fixed between $\alpha$ and $\beta$).

As in the case $d=2$,  our approach detects a discrepancy in  the behaviour of $\tX_{\ell,m}^\DimD$, depending on whether 
 $m$ is smaller or larger than $\epsilon\ell$ for some fixed $\epsilon\in (0,1)$. 
This corresponds to the fact that, if $m \leq \epsilon \ell$,  the functions in \eqref{eq:Xtilde}
are asymptotically related to Bessel functions,  
 while for $m \geq \epsilon \ell$ 
their asymptotical behaviour is described by Hermite polynomials.
Indeed  a crucial tool in the proof of our pointwise   bounds for  $\tX_{\ell,m}^\DimD$  is provided by
the precise asymptotic approximations of ultraspherical polynomials in terms of Bessel functions and Hermite polynomials 
previously obtained by Boyd and Dunster and by Olver  \cite{BoydDunster, Olver}. We point out that estimates for Hermite and Bessel functions of a similar character to those considered here are available in the literature (see, e.g., \cite{AW,BarceloRuizVega}), but they apply to one-parameter families; in contrast, here we obtain uniform estimates for two-parameter families of ultraspherical polynomials. Similarly, but in a different context, \cite{DM} presents a robust approach that applies to orthonormal expansions associated to second-order ODE on the real line, yielding estimates that are uniform with respect to an additional scale parameter.

Parts of the proofs presented here are similar 
 to those given in \cite[Section 3]{CaCiaMa}, but several variations and new ideas are required when $d>2$. 
As a matter of fact, even in the case $d=2$, here we obtain a substantially stronger decay beyond the transition point in the Hermite regime compared to the one proved in \cite{CaCiaMa}. When comparing results, one should take into account a slight change of notation, since $\ell$ in \cite{CaCiaMa} corresponds to $\ell-1/2$ here.

Let us introduce, for all $d \in \NN$, $d \geq 2$, the index set
\begin{equation}
\label{eq:Id}
I_\DimD = \{ (\ell,m) \tc \ell \in \NN_\DimD, \, m \in \NN_{\DimD-1}, \,  \ell \geq m \}.
\end{equation}
Moreover, for all $\ell,m \in \NN/2$ with $\ell \neq 0$ and $0 \leq m \leq \ell$,
we define the points $a_{\ell,m},b_{\ell,m}\in [0,1]$ by
\begin{equation}\label{eq:def-bldm}
b_{\ell,m} = \frac{m}{\ell}
\end{equation}
and
\begin{equation}\label{eq:def-aldm}
a_{\ell,m}^2 =1-b_{\ell,m}^2=\frac{(\ell-m)(\ell+m)}{\ell^2}.
\end{equation}
One should think of $\pm a_{\ell,m}$ as the values of $x \in [-1,1]$ corresponding to the transition points for $\tX_{\ell,m}^\DimD(x)$, while $b_{\ell,m}$ corresponds to the transition points after the change of variables $y = \sqrt{1-x^2}$.

In the statement below, and throughout the paper, for two given nonnegative quantities $A$ and $B$, we use the notation ``$A \lesssim B$'' to indicate that $A \leq C B$ for some positive constant $C$. We also write $A \simeq B$ as shorthand for $A \lesssim B$ and $B \lesssim A$. Variants such as $\lesssim_k$ and $\simeq_k$ are used to indicate that the implicit constants may depend on the parameter $k$.

\begin{theorem}\label{thm:main}
Let $\DimD \in \NN$, $\DimD \geq 2$.
For all $\epsilon \in (0,1)$, there exists $c \in (0,1)$ such that, for all $(\ell,m) \in I_\DimD$,
if $m \geq \epsilon \ell$, then
\begin{equation}\label{eq:main_olver}
|\tX_{\ell,m}^\DimD(x)| \lesssim_{\DimD,\epsilon} \begin{cases}
(\ell^{-1} + |x^2-a_{\ell,m}^2|)^{-1/4} &\text{for all $x\in[-1,1]$,}\\
|x|^{-1/2} (1-x^2)^{(c\ell-(\DimD-2)/4)_+} &\text{for $|x| \geq 2 \, a_{\ell,m}$.}
\end{cases}
\end{equation}
while, 
 if $m \leq \epsilon \ell$, then
\begin{equation}\label{eq:main_boyddunster}
|\tX^\DimD_{\ell,m}(x)|
\lesssim_{\DimD,\epsilon}  \begin{cases}
y^{-(\DimD-2)/2} \left( \ell^{-2} (1+m)^{4/3} + |y^2-b_{\ell,m}^2|\right)^{-1/4} & \text{for all $x \in[-1,1]$,}\\
\ell^{(\DimD-1)/2} \, 2^{-m} &\text{if $y \leq b_{\ell,m}/(2e)$,}
\end{cases}
\end{equation}
where $y = \sqrt{1-x^2}$.
\end{theorem}

The above estimates will be derived from a series of bounds for the $\DimD$-independent functions $Y_{\ell,m}$ stated in Propositions \ref{prp:est_boyddunsterEXP}, \ref{prp:est_boyddunsterV}, \ref{prp:est_olver_Legendre}, and \ref{prp:titchmarsh_agmon}. It is important to remark that the dependence on $\DimD$ of the above estimates is not only due to the factor $(1-x^2)^{-\DimD/4}$ in \eqref{eq:Xtilde}, but also to the range of indices $I_\DimD$.

\section{Notation and preliminaries}\label{s:notation}
By the symbol $P^{(\alpha,\beta)}_j$ we shall denote the Jacobi polynomial of degree $j \in \NN$ and indices $\alpha,\beta>-1$, defined by means of Rodrigues' formula:
\begin{equation*}
P^{(\alpha,\beta)}_j(x)= \frac{(-1)^j}{2^j\, j!} (1-x)^{-\alpha} (1+x)^{-\beta} \left(\frac{d}{dx}\right)^j \left((1-x)^{\alpha+j} (1+x)^{\beta+j} \right)
\end{equation*}
for $x \in (-1,1)$.
We recall, in particular, the symmetry relation 
\begin{equation*}
P_j^{(\alpha,\beta)}(x) = (-1)^j P_j^{(\beta,\alpha)}(x),
\end{equation*}
for $j \in \NN$, $\alpha,\beta>-1$ and $x \in \RR$. 

In the case $\alpha=\beta$, Jacobi polynomials reduce to ultraspherical polynomials \cite[(4.7.1)]
{Szego}.
In particular, by using the relation between Jacobi polynomials and associated Legendre functions (Ferrers functions), namely,
\[
P^{(\alpha,\alpha)}_k(x) = \frac{2^\alpha \Gamma(\alpha+k)}{k!} (1-x^2)^{-\alpha/2} \LegF_{\alpha+k}^{-\alpha}(x)
\]
for $x \in (-1,1)$, $k \in \NN$, $\alpha \geq 0$ (see \cite[formulas 14.3.1, 14.3.3, 15.8.1 and 18.5.7]{DLMF}), we can write the functions $\tX_{\ell,m}^\DimD$ as follows:
\begin{equation}\label{eq:tX_ferrers}
\tX_{\ell,m}^\DimD(x) = \sqrt{\frac{\ell\,  \Gamma (\ell+m+1/2)} {\Gamma(\ell-m+1/2) }} (1-x^2)^{-(\DimD-2)/4} \LegF^{-m}_{\ell-1/2}(x).
\end{equation}

Let now $I = \{ (\ell,m) \in (\NN/2)^2 \tc \ell-m-1/2 \in \NN\}$.
For $(\ell,m) \in I$, define
\begin{equation}\label{def:Ylm}
\begin{split}
Y_{\ell,m}(x)
&= 
c_{\ell m}
 (1-x^2)^{m/2} P_{\ell-m-1/2}^{(m,m)}(x) \\
&= \sqrt{\frac{\ell\,  \Gamma (\ell+m+1/2)} {\Gamma(\ell-m+1/2) }} \LegF^{-m}_{\ell-1/2}(x).
\end{split}
\end{equation}
Note that, if $d \geq 2$ and $m \in \NN_{d-1}$, then
\begin{equation}\label{eq:rel-Y-X}
\tX_{\ell,m}^\DimD(x) = (1-x^2)^{-(\DimD-2)/4} Y_{\ell,m}(x). 
\end{equation}

\section{Results from representation theory}
We recall some well known facts concerning the spectral theory of  the Laplace--Beltrami operator
$\Delta_\DimD$ 
 on the unit sphere $\sfera^\DimD$ in $\RR^{1+\DimD}$.
 For a detailed account of the theory  we refer to  \cite[Ch.\ 4]{Stein-Weiss} or \cite[Ch.\ 5]{AxBR}.
  
 The operator $\Delta_\DimD$ is essentially self-adjoint on $L^2(\sfera^\DimD)$, with discrete spectrum.
The symbol  $\Harm^\ell(\sfera^\DimD)$ will denote  the eigenspace of $\Delta_\DimD$
corresponding to the eigenvalue 
\begin{equation}\label{eq:laplace_eigenvalues}
\lambda_\ell^\DimD \defeq ( \ell+(\DimD-1)/2 ) ( \ell-(\DimD-1)/2 ),
\end{equation}
where $\ell\in\NN_\DimD$.
It is well-known that  
 $\Harm^\ell(\sfera^\DimD)$ consists of all spherical harmonics of degree $\ell' = \ell-(\DimD-1)/2 \in \NN$, that is, of all restrictions to $\sfera^\DimD$ of homogeneous harmonic polynomials on $\RR^{1+\DimD}$ of degree $\ell'$.

The following  facts on the spaces $\Harm^\ell(\sfera^\DimD)$  are standard.
\begin{enumerate}\item
Since $\Delta_\DimD$ is self-adjoint,
its eigenspaces are mutually orthogonal in $L^{2}(\sfera^\DimD)$, i.e.,
\[
\Harm^{\ell_1}(\sfera^\DimD) \perp \Harm^{\ell_2}(\sfera^\DimD)
\]
for $\ell_1,\ell_2 \in \NN_\DimD$, $\ell_1 \neq \ell_2$.
\item
Each $\Harm^\ell(\sfera^\DimD)$
is a finite-dimensional space of dimension
\begin{equation}\label{eq:dimensione}
{\dim(\Harm^\ell(\sfera^d))}=\binom{\ell'+\DimD}{\ell'}-\binom{\ell'+\DimD-2}{\ell'-2} 
= \frac{2\ell'+\DimD-1}{\DimD-1} \binom{\ell'+\DimD-2}{\DimD-2}
\end{equation}
for $\ell = \ell'+(d-1)/2 \in\NN_\DimD$ (the last identity in \eqref{eq:dimensione} only makes sense when $\DimD>1$).
In particular
\begin{equation}\label{eq:dimensione_stima}
{\dim(\Harm^\ell(\sfera^d))} \simeq_\DimD \ell^{\DimD-1}
\end{equation}
Here and subsequently, we adhere to the convention that $0^0 = 1$, so that 
this estimate is also valid when $\DimD=1$.
\item The spaces $\Harm^\ell(\sfera^\DimD)$ are $\group{O}(n+1)$-invariant for every $\ell\in\NN_\DimD$.
\item The representation of $\group{O}(n+1)$ on the space $\Harm^\ell(\sfera^\DimD)$
is irreducible.
\end{enumerate}

Next, we introduce  a system of ``cylindrical coordinates'' on $\sfera^\DimD$, $\DimD \geq 2$.
For all $\omega \in \sfera^{\DimD-1}$ and 
$x\in [-1,1]$,
one defines the point $\spnt{x}{\omega} \in \sfera^\DimD$ as
\begin{equation}\label{eq:coordsfera}
\spnt{x}{\omega} = (\sqrt{1-x^2}\, \omega, x).
\end{equation}
Then \eqref{eq:coordsfera} yields  a ``system of coordinates'' on $\sfera^\DimD$, modulo null sets, since, apart from $x= \pm 1$, the map $(\omega,x) \mapsto \spnt{\omega}{x}$ is a diffeomorphism onto its image, which is the sphere with the  two poles removed.

In these coordinates, the spherical measure $\meas_\DimD$ on $\sfera^\DimD$ is given by
\[
d\meas_\DimD(\spnt{\omega}{x}) =(1-x^2)^{(\DimD-2)/2} \,dx \,d\meas_{\DimD-1}(\omega),
\]
where $\meas_{\DimD-1}$ is the spherical measure on $\sfera^{\DimD-1}$. We recall that
\begin{equation}\label{eq:mis-sfera}
\meas_\DimD(\sfera^d) = \frac{(d+1)\pi^{(d+1)/2}}{\Gamma((d+3)/2)}. 
\end{equation}

The following formula, proved in \cite[Ch.\ 4, Corollary 2.9]{Stein-Weiss}, will be repeatedly used 
throughout the paper: if  $E^\DimD_\ell$ is any orthonormal basis of $\Harm^\ell(\sfera^\DimD)$,
then
\begin{equation}\label{eq:SW-dim}
\sum_{Z \in E^\DimD_\ell} | Z(z) |^2 = \meas_\DimD(\sfera^\DimD)^{-1} \, {\dim(\Harm^\ell(\sfera^d))}
\end{equation}
for all $z \in \sfera^\DimD$.

The above-mentioned properties as a whole imply
 a universal bound for $Y_{\ell,m}(x)$,  which will be useful, in particular, in the Bessel regime.
\begin{proposition}
For all $(\ell,m) \in I$
 and all $x\in [-1,1]$,
\begin{equation}\label{eq:universal-bound-Y}
 Y_{\ell,m}(x)^2 \lesssim (1-x^2)^{m}  \frac{\ell}{\sqrt{m+1}} \binom{\ell-1/2+m}{2m}. 
\end{equation}
Moreover
\begin{equation}\label{eq:universal-bound-Y-2}
Y_{\ell,m}(x)^2 \lesssim \begin{cases} \ell^{1/2} & \text{if } m \in \NN,\\
(1-x^2)^{1/2} \ell/m^{1/2} & \text{if } m \in \NN+1/2.
\end{cases}
\end{equation}
\end{proposition}

\begin{proof}
Let $\ell \in \NN_d$, $d \geq 2$.
By the decomposition \eqref{eq:decomp-Hl},
if $K^d_\ell : \sfera^d \times \sfera^d \to \RR$ is the integral kernel of the orthogonal projection of $L^2(\sfera^d)$ onto $\Harm^\ell(\sfera^d)$, then
\[
K^d_{\ell}(\spnt{x}{\omega},\spnt{x'}{\omega'}) = \sum_{\substack{m \leq \ell \\ m \in \NN_{d-1}}}\tX^d_{\ell,m}(x) \tX^d_{\ell,m}(x') K^{d-1}_{m}(\omega,\omega').
\]
Hence, in light of \eqref{eq:SW-dim},
\begin{equation}\label{eq:id_projection}
\frac{\dim(\Harm^\ell(\sfera^d))}{\meas_d(\sfera^d)} = \sum_{\substack{m \leq \ell \\ m \in \NN_{d-1}}} \tX^d_{\ell,m}(x)^2 \frac{\dim(\Harm^m(\sfera^{d-1}))}{\meas_{d-1}(\sfera^{d-1})}
\end{equation}
and in particular
\begin{equation}\label{eq:est_Y_dim}
Y_{\ell,m}(x)^2 = (1-x^2)^{(\DimD-2)/4} \tX^d_{\ell,m}(x)^2 \leq (1-x^2)^{(\DimD-2)/2} \frac{\dim(\Harm^\ell(\sfera^d))}{\dim(\Harm^m(\sfera^{d-1}))} \frac{\meas_{d-1}(\sfera^{d-1})}{\meas_d(\sfera^d)} 
\end{equation}
for all $(\ell,m) \in I_d$. Now, for a given $(\ell,m) \in I$, the estimates \eqref{eq:universal-bound-Y} and \eqref{eq:universal-bound-Y-2} follow from \eqref{eq:est_Y_dim} by choosing $\DimD \geq 2$ to be, respectively, the largest and the smallest possible so that $(\ell,m) \in I_\DimD$, and using \eqref{eq:mis-sfera} and \eqref{eq:dimensione}.
\end{proof}

\section{The Bessel regime}

In this section we prove some pointwise estimates    for $Y_{\ell,m}$
and $\tX_{\ell,m}^\DimD$ in the range $m \leq \epsilon \ell$,  for some $\epsilon \in (0,1)$.

First, from the bound \eqref{eq:universal-bound-Y} we readily derive an estimate that is particularly effective in the region where $y = \sqrt{1-x^2} \ll b_{\ell,m}$.

\begin{proposition}\label{prp:est_boyddunsterEXP}
Let $\epsilon \in (0,1)$. For all $(\ell,m) \in I$  such that
  $m \leq \epsilon \ell$, and for all $x \in[-1,1]$,
\begin{equation}\label{eq:est_boyddunster_exp}
\big| Y_{\ell,m}(x) \big| 
\lesssim_\epsilon
b_{\ell,m}^{-(m+1/2)} (ye)^m,
\end{equation}
where $y = \sqrt{1-x^2}$.
\end{proposition}
\begin{proof}
 For $m=0$ the estimate is trivial, so we may assume $m>0$.
  The universal bound  \eqref{eq:universal-bound-Y}
 implies that
  for  all $x \in [0,1]$ and all $(\ell,m) \in I$, with 
	$0 < m \leq \epsilon \ell$,
   \begin{align*}
 Y_{\ell,m}(x)^2
 &\lesssim y^{2m} \,  \frac{\ell}{\sqrt{m}} \,  \binom{\ell-1/2+m}{2m}\\
  &\lesssim_\epsilon y^{2m} \,  \frac{\ell}{\sqrt{m}} \,    \frac{1}{\sqrt{2\pi(2m)}} \,\Big(  \frac{(\ell-1/2+m)e}{2m}\Big)^{2m}\\
  &\lesssim y^{2m} \,  \frac{\ell}{{m}} \,\Big(  \frac{ \ell\,e}{m}\Big)^{2m},
\end{align*}   
as a consequence of   Stirling's approximation.  This proves  
\eqref{eq:est_boyddunster_exp}.
\end{proof}

A more precise estimate in the region where $y \gtrsim b_{\ell,m}$ can be derived from a uniform asymptotic approximation
for the associated Legendre functions $\LegF^{-m}_{\ell-1/2}$
 in terms of Bessel functions,
previously proved in
\cite{BoydDunster}. This was shown in \cite[Proposition 3.5]{CaCiaMa} in the case where $m$ is integer. The case where $m$ is half-integer can be treated similarly, however the proof requires a number of modifications, mainly due to the fact that the proof in \cite{CaCiaMa} exploits certain estimates for spherical harmonics on $\sfera^2$ from \cite{BDWZ}, which do not directly apply to the case where $m$ is not an integer. The proof presented below, instead, applies irrespective of whether $m$ is integer, and exploits the following bound from \cite{La} for the Bessel function of the first kind $J_\nu$  of order $\nu \in (-1,\infty)$.

\begin{lemma}\label{lem:besselestimate}
There exists $b \in (0,1)$ such that, for all $\nu \in (0,\infty)$ and $z \in \RR$,
\[
|J_\nu(z)| \leq b \nu^{-1/3}.
\]
\end{lemma}

By combining this bound with the results of  \cite{BoydDunster} we can prove the following estimate.

\begin{proposition}\label{prp:est_boyddunsterV}
Let $\epsilon \in (0,1)$. 
The following bounds hold for all $(\ell,m) \in I$  such that
  $m \leq \epsilon \ell$, and for all $x \in[-1,1]$:
\begin{equation}\label{eq:est_boyddunster}
\big| Y_{\ell,m}(x) \big| 
\lesssim_\epsilon
\left( \frac{(1+m)^{4/3}}{\ell^{2}} + |y^2-b_{\ell,m}^2|\right)^{-1/4} ,
\end{equation}
where $y = \sqrt{1-x^2}$.
\end{proposition}
\begin{proof}
Without loss of generality we may assume $x \geq 0$.
Following the proof of \cite[Proposition 3.5]{CaCiaMa}, by using the results of \cite{BoydDunster} we can write
\begin{equation}\label{eq:boyddunster_Xapprox}
\begin{split}
 |y^2-b_{\ell,m}^2|^{1/4} \,Y_{\ell,m}(x) 
&= \tilde \varkappa_{\ell,m} |\ell^2 \zeta_{\ell,m}(x)-m^2|^{1/4} \\
&\times \bigl[J_m(\ell \,\zeta_{\ell,m}(x)^{1/2}) \\
&+ E_m^{-1} M_m(\ell \,\zeta_{\ell,m}(x)^{1/2}) \, \bigO(\ell^{-1}) \bigr],
\end{split}
\end{equation}
uniformly in $x \in [0,1]$ and $(\ell,m) \in I$ with $m \leq \epsilon \ell$.
Here $y = \sqrt{1-x^2}$ and $\tilde \varkappa_{\ell,m} \simeq 1$
uniformly in $(\ell,m) \in I$; 
moreover,
$E_m^{-1} M_m$ is
the pointwise quotient of the auxiliary functions $M_m$ and $E_m$ introduced 
in \cite[\S 3]{BoydDunster} and $\zeta_{\ell,m} : [0,1] \to [0,\zeta_{\ell,m}(0)]$ is the decreasing bijection satisfying $\zeta_{\ell,m}(a_{\ell,m}) = b_{\ell,m}^2$ and implicitly defined by
\begin{align}
\int_{b_{\ell,m}^2}^{\zeta_{\ell,m}(x)} \frac{(\xi-b_{\ell,m}^2)^{1/2}}{2\xi} \,d\xi &= \int_x^{a_{\ell,m}} \frac{(a_{\ell,m}^2-s^2)^{1/2}}{1-s^2} \,ds \qquad\text{($0 \leq x \leq a_{\ell,m}$),} \label{eq:def_BDzeta_upper}\\
\int_{\zeta_{\ell,m}(x)}^{b_{\ell,m}^2} \frac{(b_{\ell,m}^2-\xi)^{1/2}}{2\xi} \,d\xi &= \int_{a_{\ell,m}}^x \frac{(s^2-a_{\ell,m}^2)^{1/2}}{1-s^2} \,ds \qquad\text{($a_{\ell,m} \leq x \leq 1$).} \label{eq:def_BDzeta_lower}
\end{align}
Notice that $\ell$ in \cite{CaCiaMa}
corresponds to $\ell-1/2$ here.

The same argument as in \cite{CaCiaMa} (see formula (3.20) there) shows that the right-hand side of \eqref{eq:boyddunster_Xapprox} is uniformly bounded, thus yielding that
\begin{equation}\label{eq:boyddunster_bound_X_bis}
| Y_{\ell,m}(x)| \lesssim_\epsilon \, |y^2-b_{\ell,m}^2|^{-1/4},
\end{equation}
 uniformly in $x \in [0,1]$ and $(\ell,m) \in I$ with $m \leq \epsilon \ell$.
Hence the proof of    
\eqref{eq:est_boyddunster}
will be complete if we show that
\begin{equation}\label{eq:est_boyddunster_unif}
 |Y_{\ell,m}(x)| \lesssim_\epsilon \ell^{1/2} (1+m)^{-1/3}
\end{equation}
for all $(\ell,m) \in I$ with $m \leq \epsilon \ell$ and $x \in [0,1]$.
Actually, we need only consider the case where
$b_{\ell,m}/2 \leq y \leq b_{\ell,m}(1+\delta m^{-2/3})$ for some $\delta>0$, for otherwise \eqref{eq:est_boyddunster_unif} easily follows from \eqref{eq:boyddunster_bound_X_bis}.
In this case, $y \simeq m/\ell$, and therefore $|Y_{\ell,m}(x)| \lesssim \ell^{1/2}$ by \eqref{eq:universal-bound-Y-2}; 
hence, in proving \eqref{eq:est_boyddunster_unif}, we need only consider $m \geq m_0$  for some $m_0 > 0$.

Now, as discussed in \cite[\S 3]{BoydDunster}, the identity
\[
E^{-1}_m M_m(z) = \sqrt{2} J_m(z)
\]
holds for all $z \in [0,X_m]$, where $X_m$ is a positive real number defined in \cite[eq.\ (3.4)]{BoydDunster} and satisfying
\begin{equation}\label{eq:lb_Xm_1}
X_m \geq m
\end{equation}
for all $m \geq 0$ by \cite[Corollary 1 applied with $\theta = 3\pi/4$]{MuSp}, as well as
\[
X_m = m + 2 c m^{1/3} + \bigO(m^{-1/3})
\]
as $m \to \infty$, for some $c \in (0,1)$ \cite[Chapter 12, Ex.\ 1.1, p.\ 438]{Olver-libro}.
In particular
\begin{equation}\label{eq:lb_Xm_2}
X_m \geq m(1 + c m^{-2/3})
\end{equation}
for all $m \geq m_0$, for a suitable $m_0 > 0$. Moreover, \eqref{eq:boyddunster_Xapprox} implies that
\begin{equation}\label{eq:boyddunster_bound_X_better}
 |y^2-b_{\ell,m}^2|^{1/4} \, |Y_{\ell,m}(x)| \\
\lesssim_\epsilon |\ell^2 \zeta_{\ell,m}(x)-m^2|^{1/4}
|J_m(\ell \,\zeta_{\ell,m}(x)^{1/2}) |
\end{equation}
uniformly for all $(\ell,m) \in I$ with $m \leq \epsilon \ell$ and $x \in [0,1]$ satisfying $\ell \, \zeta_{\ell,m}(x)^{1/2} \leq X_m$.

We now recall from \cite[eq.\ (3.24)]{CaCiaMa} the inequality
\begin{equation}\label{eq:boyddunster_claim}
\zeta_{\ell,m}(x)^{1/2} \leq y
\end{equation}
for all $x \in [a_{\ell,m},1]$. 
Further, we claim that
\begin{equation}\label{eq:boyddunster_claim_2}
\frac{\zeta_{\ell,m}(x) - b_{\ell,m}^2}{y^2-b_{\ell,m}^2} \simeq_\epsilon 1
\end{equation}
for all $x \in [0,1]$ with $b_{\ell,m}/2 \leq y \leq \epsilon^{-1/2} b_{\ell,m}$.

Assuming the claim, from \eqref{eq:boyddunster_claim_2} we deduce that,
for all $(\ell,m) \in I$ and $x \in [0,1]$, if $m \leq \epsilon \ell$ and $b_{\ell,m}/2 \leq y \leq b_{\ell,m}(1+\delta m^{-2/3})$ for some $\delta \in (0,1)$, then
\[
\zeta_{\ell,m}(x) \leq b_{\ell,m}^2(1 + c_\epsilon \delta m^{-2/3}),
\]
whence, by \eqref{eq:lb_Xm_2},
\[
\ell \zeta_{\ell,m}(x)^{1/2} \leq m (1 + c_\epsilon \delta m^{-2/3}) \leq X_m
\]
provided $\delta$ is chosen sufficiently small and $m \geq m_0$ for some sufficiently large $m_0$. Therefore, from \eqref{eq:boyddunster_claim_2} and \eqref{eq:boyddunster_bound_X_better} and Lemma \ref{lem:besselestimate} we deduce that
\[
 |Y_{\ell,m}(x) | \\
\lesssim_\epsilon \ell^{1/2}
m^{-1/3}
\]
for all $(\ell,m) \in I$ and $x \in [0,1]$ satisfying $m_0 \leq m \leq \epsilon \ell$ and $b_{\ell,m}/2 \leq y \leq b_{\ell,m}(1+\delta m^{-2/3})$. This completes the proof of \eqref{eq:est_boyddunster_unif}.
 
 We are left with the proof of the claim 
\eqref{eq:boyddunster_claim_2}.
Assume first that $b_{\ell,m} \leq y \leq \epsilon^{-1/2} b_{\ell,m}$. Then, by \eqref{eq:boyddunster_claim}, $b_{\ell,m} \leq \zeta_{\ell,m}^{1/2}(x) \leq \epsilon^{-1/2} b_{\ell,m}$ as well, and moreover $\sqrt{1-\epsilon^{1/2}} \leq x \leq a_{\ell,m} \leq 1$ (here we use that $b_{\ell,m} \leq \epsilon$). Consequently, from \eqref{eq:def_BDzeta_upper} we deduce that
\begin{equation}\label{eq:claim_2_proof_1}
\int_{b_{\ell,m}^2}^{\zeta_{\ell,m}(x)} (\xi - b_{\ell,m}^2)^{1/2} \,d\xi \simeq_\epsilon \int^{a_{\ell,m}}_x (a_{\ell,m}^2 - s^2)^{1/2} \,ds \simeq_\epsilon \int^{a_{\ell,m}^2}_{x^2} (a_{\ell,m}^2 - t)^{1/2} \,dt,
\end{equation}
that is,
\begin{equation}\label{eq:claim_2_proof_2}
(\zeta_{\ell,m}(x) - b_{\ell,m}^2)^{3/2} \simeq_\epsilon (a_{\ell,m}^2 - x^2)^{3/2} = (y^2-b_{\ell,m}^2)^{3/2},
\end{equation}
which gives \eqref{eq:boyddunster_claim_2} in this case. In the case where $b_{\ell,m}/2 \leq y \leq b_{\ell,m}$, instead, by \eqref{eq:def_BDzeta_lower} we first deduce that
\begin{multline*}
\frac{b_{\ell,m}}{2\sqrt{2}} \log_+\left(\frac{b_{\ell,m}^2}{2\zeta_{\ell,m}(x)}\right) \leq \int_{\min\{\zeta_{\ell,m}(x),b_{\ell,m}^2/2\}}^{b_{\ell,m}^2/2} \frac{(b_{\ell,m}^2-\xi)^{1/2}}{2\xi} \,d\xi \\
\leq \frac{4}{b_{\ell,m}^2} \int_{a_{\ell,m}}^x (s^2-a_{\ell,m}^2)^{1/2} \,ds
\simeq_\epsilon b_{\ell,m}^{-2} (x^2-a_{\ell,m}^2)^{3/2} \lesssim b_{\ell,m}
\end{multline*}
(here we used that $1 \geq x \geq a_{\ell,m} \geq \sqrt{1-\epsilon^2}$),
whence
\[
c_\epsilon b_{\ell,m} \leq \zeta_{\ell,m}(x)^{1/2} \leq b_{\ell,m}
\]
for some $c_\epsilon \in (0,1)$. Now the analogues of \eqref{eq:claim_2_proof_1} and \eqref{eq:claim_2_proof_2} can be derived by using \eqref{eq:def_BDzeta_lower} in place of \eqref{eq:def_BDzeta_upper}, giving \eqref{eq:boyddunster_claim_2} in this case as well.
\end{proof}

Propositions \ref{prp:est_boyddunsterEXP} and \ref{prp:est_boyddunsterV} immediately
yield the second part of Theorem \ref{thm:main}.

\begin{corollary}\label{thm:main_1}
Let $\DimD \in \NN$, $\DimD \geq 2$,
and  $\epsilon \in (0,1)$.
For all $(\ell,m) \in I_\DimD$, if $m \leq \epsilon \ell$, then
\begin{equation}\label{eq:main_boyddunster_1}
|\tX^\DimD_{\ell,m}(x)|
\lesssim_{\epsilon,\DimD} \begin{cases}
y^{-(\DimD-2)/2} \left( \frac{(1+m)^{4/3}}{\ell^{2}} + |y^2-b_{\ell,m}^2|\right)^{-1/4} & \text{for all $x \in[-1,1]$,}\\
2^{-m}\, \ell^{(\DimD-1)/2}  &\text{if $y \leq b_{\ell,m}/2e$,}
\end{cases}
\end{equation}
where $y = \sqrt{1-x^2}$.
\end{corollary}
\begin{proof}
 The first inequality is an immediate consequence of  \eqref{eq:rel-Y-X} and 
 \eqref{eq:est_boyddunster}. 
Moreover, if $m \in \NN_{\DimD-1}$ and $y \leq b_{\ell,m}/2e$, then
\[
\big|\tX_{\ell,m}^\DimD(x) \big|
\lesssim_\epsilon 
\big(b_{\ell,m}/2e \big)^{m-(\DimD-2)/2} b_{\ell,m}^{-m-1/2} e^m
\lesssim_d 
 2^{-m} \ell^{(\DimD-1)/2},
\]
proving the second bound in \eqref{eq:main_boyddunster_1}.
\end{proof}

\section{The Hermite regime}\label{s:boundsHermite}

In this section we prove pointwise estimates   for both $Y_{\ell,m}$
and $\tX_{\ell,m}^\DimD$ as $m \geq \epsilon \ell$ for some $\epsilon \in (0,1)$.
In this range, we can apply a uniform asymptotic approximation of  $\LegF^{-m}_{\ell-1/2}$ for large $\ell$ in terms of Hermite functions 
previously proved by Olver \cite{Olver75,Olver}.
Indeed, the same argument used in the proof of \cite[Proposition 3.3]{CaCiaMa}, which is based on Olver's approximation, as well as standard estimates for Hermite functions \cite{AW,Th} and the uniform estimate for Jacobi polynomials of Haagerup and Schlichtkrull \cite{Haagerup}, can be applied to prove the following estimate.

\begin{proposition}\label{prp:est_olver_Legendre}
Let $\epsilon \in (0,1)$. 
Then
for all $(\ell,m) \in I$ with $m \geq \epsilon \ell$ and
for all $x\in[-1,1]$
\begin{equation}\label{eq:est_olver}
 \big| Y_{\ell,m}(x)\big|
 \lesssim_\epsilon \big(\ell^{-1}+  |x^2-a_{\ell,m}^2|\big)^{-1/4} .
\end{equation}
\end{proposition}

By combining this estimate with ODE techniques we can obtain a stronger decay estimate in the region where $|x| \gg a_{\ell,m}$.	

\begin{proposition}\label{prp:titchmarsh_agmon}
For all $K \in (1,\infty)$ there exists $c \in (0,1)$ such that, 
for all $\epsilon \in (0,1)$ and $m_0 \in \NN/2$,
if $(\ell,m) \in I$ is such that $m \geq \max\{\epsilon \ell,m_0\}$, then 
 \begin{equation}\label{eq:m0_c}
|Y_{\ell,m}(x)| \lesssim_{\epsilon,m_0,K} |x|^{-1/2} (1-x^2)^{\max\{c\epsilon\ell,m_0\}/2}
 \end{equation}
whenever $x \in (-1,1)$ and $|x| \geq K a_{\ell,m}$.
\end{proposition}
\begin{proof}
Note that, if $m \leq 1$, then $\ell \lesssim_\epsilon 1$ and the desired estimate trivially follows from \eqref{eq:universal-bound-Y}. So in what follows we may assume $m>1$. For a similar reason, we may also assume that $\ell \geq \ell(m_0,K)$ for some large $\ell(m_0,K)$ to be specified later. Further, due to parity, we need only prove the estimate for $x\geq 0$.

Recall (see, e.g., \cite[eq.\ (2.1)]{Olver}) that the function $L(x) = (1-x^2)^{1/2} Y_{\ell,m}(x)$ satisfies the ODE
\begin{equation}\label{eq:L''}
L''(x) = Q(x) L(x) 
\end{equation}
on the interval $(-1,1)$,
where
\begin{equation}\label{eq:defQ}
Q(x)= Q_{\ell,m}(x) = \frac{\ell^2(x^2-a_{\ell,m}^2)-(3+x^2)/4}{(1-x^2)^2} 
= (\ell^2-1/4) \frac{x^2-\bar x_{\ell,m}^2}{(1-x^2)^2},
\end{equation}
with $a_{\ell,m}$
 defined as in \eqref{eq:def-aldm}, and 
\begin{equation}\label{eq:xlm}
\bar x_{\ell,m} = \sqrt{\frac{\ell^2-m^2+3/4}{\ell^2-1/4}}
 \in [a_{\ell,m},8a_{\ell,m}]
\end{equation}
for all $(\ell,m) \in I$.
Note that $\bar x_{\ell,m} < 1$ (since $m > 1$), and $Q(x) > 0$ whenever $|x| > \bar x_{\ell,m}$. In addition, since $m>1$, from \eqref{def:Ylm} we deduce that
\begin{equation}\label{eq:Llimits}
\lim_{x \to 1} L(x) = \lim_{x \to 1} L'(x) = 0.
\end{equation}

We now claim that $L(x) L'(x) < 0$ for all $x > \bar x_{\ell,m}$. Indeed, $L(x)$ and $L'(x)$ cannot vanish simultaneously, because $L$ is a nontrivial solution of a second order linear ODE.
Moreover, by \eqref{eq:Llimits},
$L(x) L'(x)$ cannot be positive for any $x > \bar x_{\ell,m}$ (otherwise by \eqref{eq:L''} 
the function $L$ would be positive and increasing, or negative and decreasing, on the interval $(x,1)$, and would not tend to zero).
Finally one cannot have $L(x) L'(x) = 0$ for any $x > \bar x_{\ell,m}$ (because for any larger $x$ one would find the situation that we have just ruled out).
 
 Note also that $Q$ is strictly increasing for $x \geq 0$.  We can then apply the argument in \cite[\S 8.2]{Titchmarsh} and conclude that, for $x > x_* > \bar x_{\ell,m}$,
\begin{equation}\label{ineq:Titchmarsch}
|L(x)| \leq |L(x_*)| \exp \left( -\int_{x_*}^x Q(u)^{1/2} \,du \right). 
\end{equation}
From \eqref{eq:defQ} we deduce that, if $x^2 \geq (1-\eta^2)^{-1} \bar x_{\ell,m}^2$ for some $\eta \in (0,1)$, then
\[
Q(x)^{1/2} \geq \eta \sqrt{\ell^2-1/4} \frac{x}{1-x^2},
\]
and consequently, for 
$x > x_* \geq (1-\eta^2)^{-1/2} \bar x_{\ell,m}$,
\[
\int_{x_*}^x Q(u)^{1/2} \,du \geq \frac{\eta}{2} \sqrt{\ell^2-1/4} \int_{x_*^2}^{x^2} \frac{du}{1-u} = \frac{\eta}{2} \sqrt{\ell^2-1/4} \log \frac{1-x_*^2}{1-x^2}.
\]
Hence
\eqref{ineq:Titchmarsch} yields
\[
|Y_{\ell,m}(x)| \leq |Y_{\ell,m}(x_*)| \left( \frac{1-x^2}{1-x_*^2} \right)^{(\eta\sqrt{\ell^2-1/4}-1)/2}.
\]
Note that, if we take $x^2 \geq (1-\delta)^{-1} x_*^2$ for some $\delta \in (0,1)$, then $1-x_*^2 \geq 1-(1-\delta)x^2 \geq (1-x^2)^{1-\delta}$, by Bernoulli's inequality, whence
\begin{equation}\label{eq:est_titchmarsh}
|Y_{\ell,m}(x)| \leq |Y_{\ell,m}(x_*)| \, (1-x^2)^{\delta(\eta\sqrt{\ell^2-1/4}-1)/2}.
\end{equation}

Finally, let us remark that $\ell^2-m^2 \geq (\ell+m)/2$ for all $(\ell,m)\in I$. Consequently, by \eqref{eq:xlm}, $\bar x_{\ell,m}/a_{\ell,m} \to 1$ as $\ell \to \infty$ uniformly in $m$, so there exists $\ell_{K,\eta} \in \NN/2$ such that 
\begin{equation}\label{eq:large_conditions}
\bar x_{\ell,m}/a_{\ell,m} \in [1,K^{1/3}], \qquad \eta\sqrt{\ell^2-1/4} -1 \geq \eta \ell/2.
\end{equation}
for all $(\ell,m) \in I$ with $\ell \geq \ell_{K,\eta}$.
Moreover
\begin{equation}\label{eq:alm_ell}
a_{\ell,m}^2 \geq 1/(2\ell)
\end{equation}
for all $(\ell,m) \in I$, and therefore, for any $\alpha>0$,
\begin{equation}\label{eq:swallow}
|x|/a_{\ell,m} \lesssim \ell^{1/2} |x| \lesssim_\alpha \exp(\alpha \ell x^2) \leq (1-x^2)^{-\alpha\ell}.
\end{equation}

Now, since $m \geq \epsilon\ell$, if we take $x_* = (1-\eta^2)^{-1/2} \bar x_{\ell,m}$, then $x_* \geq (1-\eta^2)^{-1/2} a_{\ell,m}$ and
\begin{equation}\label{eq:est_olver_x*}
|Y_{\ell,m}(x_*)| \lesssim_{\epsilon,\eta}  a_{\ell,m}^{-1/2}
\end{equation}
by \eqref{eq:est_olver}. 
Hence, by \eqref{eq:est_titchmarsh}, \eqref{eq:est_olver_x*} and \eqref{eq:swallow},
if $x^2 \geq (1-\delta)^{-1} (1-\eta^2)^{-1} \bar x_{\ell,m}^2$, then
\[\begin{split}
|Y_{\ell,m}(x)| &\lesssim_{\epsilon,\eta} a_{\ell,m}^{-1/2} \, (1-x^2)^{\delta(\eta\sqrt{\ell^2-1/4}-1)/2} \\
 &\lesssim_{\alpha} |x|^{-1/2} (1-x^2)^{\delta(\eta\sqrt{\ell^2-1/4}-1-\alpha\ell)/2}.
\end{split}\]
As a consequence, by \eqref{eq:large_conditions}, if we take $\delta$ and $\eta$ so that $1-\delta=1-\eta^2=K^{-2/3}$, $\alpha = \eta/4$ and $c=\delta\eta/4$, then
\[
|Y_{\ell,m}(x)|  \lesssim_{\epsilon,K} |x|^{-1/2} (1-x^2)^{c\ell/2}.
\]
whenever $x \geq Ka_{\ell,m}$, $m \geq \epsilon \ell$ and $\ell \geq \ell_{K,\eta}$. This proves the desired estimate \eqref{eq:m0_c} for all $\ell \geq \ell(m_0,K)=\max\{\ell_{K,\eta},m_0/c\}$.
\end{proof}

The first part of Theorem \ref{thm:main}
is a consequence of 
 the following result.

\begin{corollary}\label{thm:est_olver_boyddunster}
Let $\DimD \in \NN$, $\DimD \geq 2$.
For all $K \in (1,\infty)$, there exists $c \in (0,1)$ such that, for all $\epsilon \in (0,1)$, for all $(\ell,m) \in I_\DimD$, if $m \geq \epsilon \ell$ then
\begin{equation}\label{eq:est_olver_xx}
|\tX_{\ell,m}^\DimD(x)| \lesssim_{\epsilon,K,\DimD}
\begin{cases}
(\ell^{-1} + |x^2-a_{\ell,m}^2|)^{-1/4} &\text{for all } x \in [-1,1], \\
|x|^{-1/2} (1-x^2)^{(c\epsilon\ell-(d-2)/2)_+/2} &\text{if } |x| \geq K \, a_{\ell,m}.
\end{cases}
\end{equation}
\end{corollary}
\begin{proof}
In light of \eqref{eq:rel-Y-X}, the second estimate in \eqref{eq:est_olver_xx} immediately follows from Proposition \ref{prp:titchmarsh_agmon} applied with $m_0=(d-2)/2$. Let now $\bar\epsilon=(1-\epsilon^2)^{1/2}$ and note that $a_{\ell,m} \leq \bar\epsilon$ whenever $m \geq \epsilon\ell$. By Proposition \ref{prp:titchmarsh_agmon} applied with $\bar\epsilon^{-1/2}$ in place of $K$, we also deduce that
\[
|\tX_{\ell,m}(x)| \lesssim_{\epsilon,d} |x|^{-1/2} \lesssim_{\epsilon} a_{\ell,m}^{-1/2}
\]
whenever $|x| \geq \bar\epsilon^{-1/2} a_{\ell,m}$, and in particular whenever $|x| \geq \bar\epsilon^{1/2}$. In view of \eqref{eq:alm_ell}, this proves the first estimate in \eqref{eq:est_olver_xx} whenever $|x| \geq \bar\epsilon^{1/2}$. Since $\bar\epsilon \in (0,1)$, the same estimate for $|x| \leq \bar\epsilon^{1/2}$ immediately follows from Proposition \ref{prp:est_olver_Legendre} and \eqref{eq:rel-Y-X}.
 \end{proof}

%%%%%%%%%%%%%%%%%%%%%%%%%%%%%%%%%%%%%%%%%%%%%%%%%%%%%%%%%%%%%%%%%%

\end{document}